\documentclass{amsart}
\usepackage[all]{xy}

\usepackage{a4wide,amssymb,mathrsfs}

\newtheorem{Pro}{Proposition}[section]
\newtheorem{Le}[Pro]{Lemma}
\newtheorem{Th}[Pro]{Theorem}
\newtheorem{Co}[Pro]{Corollary}
\theoremstyle{definition}

\theoremstyle{remark}

\let\al\alpha

\let\x\times

\let\d\partial
\let\then\Rightarrow

\let\xto\xrightarrow

\def\0{^{[1]}}

\def\ee{_{\mathrm{ee}}}

\def\1{^{-1}}

\def\cref#1#2#3{\left(#2\right.\left|\ #3\right)_{#1}}

\def\ee{{\mathscr E}}

\def\A{{\mathbb A}}

\def\ker{{\sf Ker}}

\def\cok{{\sf Coker}}

\DeclareMathOperator\Ext{{\mathsf{Ext}}}%
\def\hom{{\sf Hom}}%
\DeclareMathOperator\Ch{{\mathsf{Ch}}}%
\def\Hom{{\bf Hom}}

\numberwithin{equation}{subsection}

\begin{document}

\title{Abelian categories versus abelian $2$-categories}

\author[T. Pirashvili]{Teimuraz  Pirashvili}
\address{
Department of Mathematics\\
University of Leicester\\
University Road\\
Leicester\\
LE1 7RH, UK} \email{tp59-at-le.ac.uk}

\dedicatory{ Dedicated to professor Nodar Berikashvili on the
occasion of his  80th birthday}

\maketitle

\section{Introduction}
Two years ago we  gave a talk on abelian 2-categories in the
Max-Planck Institute of Mathematics  in Bonn. Our approach was
motivated by the theory of categorical modules over a categorical
rings \cite{catrings}. Details of the theory of categorical modules
will soon appears in our  joint work with Vincent Schmitt.

In the recent preprint Mathieu Dupont \cite{dupont} rediscovered
this notion. Our original axioms were equivalent but not the same as
one given in  \cite{dupont}. However  we did not used (co)pips and
(co)roots. Our  approach in the subject will appear elsewhere.

The paper \cite{dupont} contains several interesting results unknown
to us, however some of the results of Dupont were known to us
including Corollary 192 \cite{dupont}, which claims that the
category of discrete and codiscrete (or connected) objects are
equivalent abelian categories. Dupont posses also a question whether
any abelian category comes in this way. We will give a rather
trivial solution of this problem in the case when a given abelian
category has enough projective or injective objects, which we have
known for several years.

In this note we follow \cite{dupont} with few exceptions. We use the
term 2-kernel and 2-cokernel for what Dupont calls kernel and
cokernel and keep the terms kernel and cokernel  in the classical
meaning.

 Groupoids arising in this note are in fact Picard categories. In
 particular $\pi_1$ did not depends on the base point and
 therefore we omitted it.

\section{The 2-category of arrows $\A\0$}\label{A02}
Let $\A$ be an abelian category and let $\A\0$ be the $2$-category
of arrows of $\A$. Recall that objects of $\A\0$ are arrows
$a:A_1\to A_0$ of $\A$, a morphism  from $a:A_1\to A_0$ to $b:B_1\to
B_0$ is a pair $(f_0,f_1)$ of morphisms in $\A$ such that the
diagram
$$
\xymatrix{A_1\ar[r]^{f_1}\ar[d]^{a}& B_1\ar[d]^{b}\\
 A_0\ar[r]^{f_0}& B_0.}
$$
commutes, while a $2$-arrow $(f_0,f_1)\then (g_0,g_1)$ is an arrow
$\al:A_0\to B_1$ in $\A$ such that
$$f_1-g_1=\al a$$
$$f_0-g_0=b\al$$
with obvious compositions. In this case we also say that $(f_0,f_1)$
is homotopic to $(g_0,g_1)$ and write $(f_0,f_1)\sim (g_0,g_1)$.

For objects $a:A_1\to A_0$ and $b:B_1\to B_0$ we let
$\Hom_{\A\0}(a,b)$ be the corresponding hom-groupoid. It is clear
that
$$\pi_1(\Hom_{\A\0}(a,b))=\hom_\A(\cok(a),\ker(b))$$
However, in general we do not have a nice description of
$\pi_0(\Hom(a,b))$ in terms of $\ker(a),\cdots ,\cok(b)$. The
situation can be improved in some particular cases. To state the
corresponding result we need an additional category $\ee$. Objects
of $\ee$ are triples $(M,N,x)$, where $M$ and $N$ are objects of the
category $\A$, while $x\in \Ext^2_\A(M,N)$. A morphism from
$(M,N,x)$ to $(M',N',x')$ is a pair $(f,g)$, where $f:M\to M'$ and
$g:N\to N'$ are morphisms in $\A$ such that the equality
$f^*(x')=g_*(x)$ holds in $\Ext^2(M,N')$. For an object $a:A_1\to
A_0$ we let $Ch(a)$ be the triple $(\cok(a),\ker(a),ch(a)$, where
$ch(a)$ is the class of the 2-fold extension
$$0\to \ker(a)\to A_1\xto{a} A_0\to \cok(a)\to 0$$
in $\Ext^2_\A(\cok(a),\ker(a))$. In this way one gets a functor
$Ch:\A\0\to \ee$. We recall the following well-known result.
\begin{Le} \label{cnobili} Let $a:A_1\to A_0$ and $b:B_1\to
B_0$  be two objects of  $\A\0$. If $A_0$ is a projective object in
$\A$, then one has an exact sequence
$$0\to\Ext^1_\A(\cok(a),\ker(b))\to \pi_0(\Hom_{\A\0}(a,b))\to
\Hom_\ee(Ch(a),Ch(b))\to 0$$
\end{Le}

\begin{Co} \label{cof}

\begin{enumerate}
\item[(i)]
 Let $a:A_1\to A_0$,  $b:B_1\to
B_0$  and $b':B'_1\to B_0'$ be objects of  $\A\0$ and let $b\to b'$
be a morphism in $\A\0$, such that the induced morphisms $\ker(b)\to
\ker(b')$ and $\cok(b)\to \cok(b')$ are isomorphisms. If $A_0$ is a
projective object in $\A$, then the induced morphism of groupoids
$$\Hom_{\A\0}(a,b)\to \Hom_{\A\0}(a,b')$$
is an equivalence of categories.

\item[(ii)] Let $a:A_1\to A_0$ and  $b:B_1\to
B_0$   be objects of  $\A\0$ such that $A_0$ and $B_0$ are
projective objects in $\A$.  If $a\to b$ is a morphism in $\A\0$,
such that the induced morphisms $\ker(a)\to \ker(b)$ and $\cok(a)\to
\cok(b)$ are isomorphisms, then $a$ and $b$ are equivalent.
\end{enumerate}
\end{Co}

\begin{proof} i) In this case we have a nice description for
$\pi_i(\Hom_{\A\0}(a,-))$, which shows that the functor
$\Hom_{\A\0}(a,b)\to \Hom_{\A\0}(a,b')$ yields an isomorphism on
$\pi_0$ and $\pi_1$ and hence is an equivalence of categories. ii)
By the same reason the induced functor $\Hom_{\A\0}(x,a)\to
\Hom_{\A\0}(x,b)$ is an equivalence of categories for all $x\in
\A\0_c$ and hence we can use the Yoneda lemma for 2-categories.
\end{proof}
\section{The 2-category $\A\0_c$} In this section we will
assume that $\A$ is an abelian category with enough projective
objects.

We let $\A\0_c$ be the full 2-subcategory of the 2-category $\A\0$
consisting of objects $a:A_1\to A_0$ such that $A_0$ is a projective
object of $\A$.
\begin{Th}\label{chemia} If $\A$ is an abelian category with enough projective
objects, then $\A\0_c$ is a 2-abelian $\sf Gpd$-category. The
subcategory of discrete and codiscrete objects are equivalent to
$\A$.
\end{Th}

The rest of this work is devoted to the proof.  The first
observation is that the direct sum in $\A$ yields an additive $\sf
Gpd^*$-category structure. The next task is to characterize
faithful, fully faithful, cofaithful and fully
cofaithful morphisms.

\begin{Le} \label{32200809} \begin{enumerate}
\item[(i)]  A  morphism $(f_0,f_1):a\to b$
is faithful in $\A\0_c$ iff the morphism
$$\begin{pmatrix}-a\\ f_1\end{pmatrix}: A_1\to A_0\oplus B_1$$
is a monomorphism in $\A$.
 \item[(ii)] Let   $(f_0,f_1):a\to b$ be a  morphism in $\A\0_c$ and
 let
$$g:\ker(a)\to \ker(b), \ \ \ h:\cok(a)\to \cok(b)$$
be  induced morphisms.  Then $(f_0,f_1)$ is fully faithful in
$\A\0_c$ iff $g$ is an isomorphism and $h$ is  a monomorphism in
$\A$.
\end{enumerate}
\end{Le}

\begin{proof} i) By definition $(f_0,f_1):a\to b$ is faithful iff the
induced homomorphism
$$\pi_1(\Hom_{\A\0}(x,a))\to \pi_1(\Hom_{\A\0}(x,b))$$
is a monomorphism for all $x\in \A\0_c$. But the homomorphism in the
question is the same as
$$\hom_\A(\cok(x),\ker(a))\to \hom_\A(\cok(x),\ker(b))$$
Since $\A$ has  enough projective objects, any object in $\A$ is
isomorphic to an object of the form $\cok(x)$ for a suitable $x\in
\A\0_c$. Thus $(f_0,f_1):a\to b$ is faithful iff the induced
homomorphism $g:\ker(a)\to \ker(b)$ is a monomorphism and the i)
follows. ii) By definition $(f_0,f_1):a\to b$ is fully faithful iff
the induced functor
$$(f_0,f_1)^x:\Hom_{\A\0}(x,a))\to \Hom_{\A\0}(x,b))$$
is full and faithful. This happens iff the functor $(f_0,f_1)^x$
yields an isomorphism on $\pi_1$ and  a monomorphism on $\pi_0$.
Thus for all $x\in \A\0_c$ the induced homomorphism
$$\hom_\A(\cok(x),\ker(a))\to \hom_\A(\cok(x),\ker(b))$$ is an
isomorphism  and the induced map $$\pi_0(\Hom_{\A\0}(x,a))\to
\pi_1(\Hom_{\A\0}(x,b))$$ is a monomorphism. From the first
condition follows that $g:\ker(a)\to \ker(b)$ is an isomorphism.
This fact and Lemma \ref{cnobili} yields that the induced map
$Ch(a)\to Ch(b)$ is a monomorphism in $\ee$. Since $g$ is an
isomorphism without loss of generality we can identify $\ker(a)$ and
$\ker(b)$. Then we have $ch(a)=h^*(ch(b))$. We set $K=\ker(h)$ and
let $i:K\to \cok(a)$ be an inclusion. We have
$i^*(ch(a))=(hi)^*(ch(b))=0$. Hence $(i,0): (K,0,0)\to Ch(a)$ is a
well defined morphism in $\ee$ which is annulated by the
monomorphism $Ch(a)\to Ch(b)$. Hence $K=0$ and the result follows.
\end{proof}

Recall that $a$ is discrete if $a\to 0$ is faithful. Hence an object
$a:A_1\to A_0$ in $\A\0_c$ is discrete iff $a$ is a monomorphism in
$\A$.
\begin{Co} The functor ${\sf Dis}(\A\0_c)\to \A$
given by $a\mapsto \cok(a)$ is an equivalence of categories.
\end{Co}

\begin{Le} \label{33200809} \begin{enumerate}
\item[(i)]  A  morphism $(f_0,f_1):a\to b$
is cofaithful iff the morphism
$$(f_0,b): A_0\oplus B_1\to B_0$$
is an epimorphism in $\A$.
 \item[(ii)] Let   $(f_0,f_1):a\to b$ be a  morphism in $\A\0_c$ and
 let
$$g:\ker(a)\to \ker(b), \ \ \ h:\cok(a)\to \cok(b)$$
be  induced morphisms.  Then $(f_0,f_1)$ is fully cofaithful in
$\A\0_c$ iff $h$ is an isomorphism and $g$ is  an epimorphism in
$\A$.
\end{enumerate}
\end{Le}

\begin{proof} By definition $(f_0,f_1):a\to b$
is cofaithful (resp. fully cofaithful) iff  the induced functor
$$(f_0,f_1)_x:\Hom_{\A\0}(b,x))\to \Hom_{\A\0}(a,x))$$
is faithful (resp. full and faithful). Since any object of $\A$ is
of the form $\ker(a)$ for a suitable $a\in \A\0_c$, it follows from
the description of $\pi_i(\Hom_{\A\0})$ given in Section \ref{A02}
(essentially by the same argument as in Lemma \ref{32200809})  that
this happens iff the map $h$ is a monomorphism (resp. $h$ is an
isomorphism and the map $Ch(a)\to Ch(b)$ is an epimorphism in
$\ee$). This already proves i) and in ii) it remains to show  $g$ is
an epimorphism. We set $C=\cok(g)$ with canonical morphism
$q:\ker(b)\to C$. Since $h$ is an isomorphism we can and we will
identify $\cok(a)$ and $\cok(b)$. Then we will have
$ch(b)=g_*(ch(a))$. Hence $q_*(ch(b))=q_*g_*(ch(a))=0$. Thus
 $:(q,0,0):\Ch(b)\to (C,0,0)$ is a well-defined morphism which is
annulated by the epimorphism $Ch(a)\to Ch (b)$. Thus $C=0$.

\end{proof}

 For an object $x:X_1\to X_0$  of $\A\0$ we choose a
projective object $P_0$ and an epimorphism $\epsilon:P_0\to X_0$ and
consider the pull-back diagram
$$\xymatrix{P_1\ar[r]^{\epsilon_1}\ar[d]_{p}&X_1\ar[d]^x\\
P_0\ar[r]_{\epsilon_0}&X_0}$$ It is clear that the induced morphisms
$\ker(p)\to \ker(x)$ and $\cok(p)\to \cok(x)$ are isomorphisms and
$p\in \A\0_c$. We call $p$  a \emph{replacement} of $x$. Sometimes
it is denoted by $x^{rep}$. The following easy Lemma shows that this
is well-defined.

\begin{Le} Let $f:a\to x$ be a morphism in $\A\0$ which induce
isomorphism on $\ker$ and $\cok$. If $a\in \A\0_c$  then $f$ has the
lifting to $x^{rep}$, which is unique up to unique homotopy.
\end{Le}

 Now we discuss 2-kernels and 2-cokernels in $\A\0_c$. Let
$(f_0,f_1):a\to b$ be a morphism in $\A\0_c$. According to
\cite{dupont} the 2-cokernel of $(f_0,f_1)$ in $\A\0$ is $(q':Q\to
B_0,(id,q),\xi)$;
$$\xymatrix{A_1\ar[r]^{f_1}\ar[d]_a&B_1\ar[r]^q\ar[d]&Q\ar[d]^{q'}\\
A_0\ar[r]_{f_0}\ar[rru]&B_0\ar[r]_{id} &B_0}$$ where
$$\xymatrix{A_1\ar[r]^{f_1}\ar[d]_a&B_1\ar[d]^q\\
A_0\ar[r]_{\xi}&Q}$$ is given by the classical push-out
construction. It follows that if $b\in \A\0_c$, then the 2-cokernel
is also in $\A\0_c$.  Hence the 2-category $\A\0_c$ has 2-cokernels
and the inclusion $\A\0_c\to \A\0$ preserves 2-cokernels.

Thanks to \cite{dupont} the 2-kernel of $(f_0,f_1):a\to b$ in $\A\0$
is $(k':A_1\to K, (k,id),\kappa)$:
$$\xymatrix{A_1\ar[r]^{id}\ar[d]^{k'}
&A_1\ar[r]^{f_1}\ar[d]&B_1\ar[d]_b\\
K\ar[r]_k\ar[rru]&A_0\ar[r]_{f_0}&B_0}$$ where
$$\xymatrix{K\ar[r]^{\kappa}\ar[d]_k&B_1\ar[d]^b\\A_0\ar[r]_{f_0}& B_0}$$
is given by the classical pull-back
construction. In general $K$ is not a projective object even if
$A_0$ and $B_0$ are projective objects in $\A$. Hence $k':A_1\to K$
does not belongs to $\A\0_c$. Let $c:C_1\to C_0$ be a replacement of
$k'$. Thus we have an epimorphism $\epsilon:C_0\to K$ with
projective $C_0$ and the pull-back diagram
$$\xymatrix{C_1\ar[r]_{\epsilon'}\ar[d]_c&A_1\ar[d]_{k'}\\
C_0\ar[r]_\epsilon &K}$$ We claim that $(c:C_1\to C_0,(k\epsilon,
\epsilon'),\kappa\epsilon)$ is a 2-kernel of $(f_0,f_1)$. Indeed, we
have to show that for any object $x$ in $\A\0_c$ the groupoids
$\Hom_{\A\0}(x,c)$ and  the 2-kernel of
$$\Hom_{\A\0}(x,a)\to \Hom_{\A\0}(x,b)$$
are equivalent. But we know that the last groupoid is equivalent to
$\Hom_{\A\0}(x,k')$. Since the morphism $(\epsilon, \epsilon'):c\to
k'$ satisfies the conditions of Corollary \ref{cof} the claim
follows.

Now is easy to see that the 2-cokernel of
$(k\epsilon,\epsilon'):c\to a$ is $k:K\to A_0$ and the 2-kernel of
$(id,q):b\to q'$ is $e:E_1\to E_0$, where $E_0\to Q$ is an
epimorphism with $E_0$ projective and $E_1$ is the pull-back
$$\xymatrix{E_1\ar[r]\ar[d]^e& B_1\ar[d]^q\\ E_0\ar[r]&Q}$$

It follows from the description of 2-kernels and 2-cokernels  that
$$\Sigma(a:A_1\to A_0)=(\cok(a)\to 0)$$
and
$$\Omega(a:A_1\to A_0)=(F_1\to F_0),$$
where
$$0\to F_1\to F_0\to A_1\xto{a} A_0\to 0$$
is an exact sequence with projective $F_0$.

Thus $Pip(f)=(r:R_1\to R_0)$, where
$$0\to R_1\to R_0\xto{\pi} A_1\to K$$
is an exact sequence with projective $R_0$ and $\pi$ can be
considered as a homotopy
$$\xymatrix{R_1\ar[r]^0\ar[d]^r& A_1\ar[d]^a\\
R_0\ar[r]_0 &A_0}$$ Since $\Sigma(r)=(\ker(k')\to 0)$ we see that
the morphism $\omega_f:Coroot(\pi)\to k'$ defined in \cite{dupont}
is an equivalence. A similar argument works also for $Copip(f)$.
This finishes the proof of Theorem \ref{chemia}.

{\bf Remarks}. 1) If $\A$ has enough injective objects, then we can
consider the dual construction. Namely, we  let $\A\0_f$ be the full
2-subcategory of the 2-category $\A\0$ consisting of objects
$a:A_1\to A_0$ such that $A_1$ is an injective object of $\A$. Then
$\A\0_f$ is a 2-abelian $\sf Gpd$-category and the category of
discrete and codiscrete objects of $\A\0_f$ are equivalent to $\A$.

 2) The same method can be used to construct another abelian
 2-categories.  Recall that a \emph{symmetric
 categorical group}   consists of the
following data
$$C=(\d:C_{ee}\to C_e,\ \ \{-,-\}:C_e\x C_e\to  C_{ee})$$  where $C_e$ and
$C_{ee}$ are groups and  $\d$ is a homomorphism, while $\{-,-\}$ is
a map such that the following equalities hold for $x,y,z\in C_e$ and
$a,b\in C_{ee}$.
$$\d \{x,y\}=x^{-1}y^{-1}xy$$
$$\{\d a, \d b\}=a^{-1}b^{-1}ab$$
$$\{\d a,x\}\{x, \d a\}=1$$
$$\{x,yz\}=\{x,z\}\{x,y\}\{y^{-1}x^{-1}yx,z\}$$
$$\{xy,z\}=\{y^{-1}xy^{-1},y^{-1}zy\}\{y,z\}.$$
$$ \{x,y\}\{y,x\}=1.$$
Symmetric categorical groups obviously form a 2-category $\sf SCG$.
We let ${\sf SCG}_c$ be the full 2-subcategory formed by objects
$C=(\d:C_{ee}\to C_e)$ with free $C_e$. In a similar manner one can
prove that ${\sf SCG}_c$ is a $2$-abelian $\sf Gpd$-category.

3) All these examples are particular cases of the following general
construction. Let $\bf M$ be a closed pointed simplicial model
category. $M$ is called \emph{additive} if the natural maps from the
coproduct to the products
$$\begin{pmatrix}1&0\\0&1\end{pmatrix},
\begin{pmatrix}1&1\\0&1\end{pmatrix}:X\vee Y\to X\times Y$$
are weak equivalences. Moreover, $\bf M$ is called two-stage if
$\Sigma^2(X)\to 0$ is a weak equivalence.  If $\bf M$ is a two-stage
additive pointed model simplicial category then under some
hypothesis ${\bf M}_{cf}$ is a $2$-abelian $\sf Gpd$-category, where
objects of ${\bf M}_{cf}$ are fibrant and cofibrant objects,
morphisms in ${\bf M}_{cf}$  are usual morphisms, while 2-arrows are
homotopy classes of homotopies. The details will be given elsewhere.

\end{document}